\documentclass[letterpaper, 10 pt, conference]{ieeeconf}
\usepackage{defs}
\begin{document}
\title{ \LARGE \bf A Moreau Envelope Approach for LQR Meta-Policy Estimation}
\author{Ashwin Aravind, Mohammad Taha Toghani, and C\'esar A. Uribe
\thanks{AA is with the Centre for Systems and Control, IIT Bombay, Mumbai, MH, India. \texttt{a.aravind@iitb.ac.in}. MTT and CAU are with the Department of Electrical and Computer Engineering, Rice University, Houston, TX, USA.  \texttt{\{mttoghani, cauribe\}@rice.edu}.} 
\thanks{AA thanks Mehta Rice Engineering Scholars Program and IoE - IIT Bombay for the support during this work. AA is supported by the Ministry of Human Resource Development, Govt. of India. MTT's research is supported by the Lodieska Stockbridge Vaughn Fellowship. This work is supported by the National Science Foundation under Grants \#2211815 and \#2213568 and the Google Scholar Research Award.}
}
\maketitle
\begin{abstract}
We study the problem of policy estimation for the Linear Quadratic Regulator (LQR) in discrete-time linear time-invariant uncertain dynamical systems. We propose a Moreau Envelope-based surrogate LQR cost, built from a finite set of realizations of the uncertain system, to define a meta-policy efficiently adjustable to new realizations. Moreover, we design an algorithm to find an approximate first-order stationary point of the meta-LQR cost function. Numerical results show that the proposed approach outperforms naive averaging of controllers on new realizations of the linear system. We also provide empirical evidence that our method has better sample complexity than Model-Agnostic Meta-Learning (MAML) approaches. 
\end{abstract}
\section{Introduction}

Complexity and uncertainty are inherent to the control of modern engineering systems and their applications. Interactions with other socio-technical systems or unpredictability in the operating environment hinder our ability to obtain accurate models as the effects of such complex interactions might become apparent only during task execution~\cite{bertsekas2012dynamic, khalil1996robust}. 

Reinforcement Learning (RL) has emerged as a powerful, data-driven approach in designing controllers, particularly in scenarios where the system model is not fully known or is too complex to be captured by traditional methods~\cite{MF-RG-SK-MM:18, BG-PME-TS:20, BP-ZPJ:21, GJ-HB-JG-AC-PKS:21, HW-LFT-AM-JA:23,giegrich2024convergence,sforni2024stability}. Such methods involve RL techniques to iteratively learn optimal control policies through interaction with the system, bypassing the need for an explicit model.

The seminal work in \cite{MF-RG-SK-MM:18} introduces a data-driven linear quadratic regulator 
(LQR) design using RL frameworks and demonstrates global convergence to an optimal control policy without requiring a model for the system. However, the uncertainties in the system while executing the control policy can significantly affect the stability and operating costs. Thus, developing resilient strategies that quickly adapt to new realizations of the uncertain system is crucial when designing control policies.

Meta-learning facilitates quick adaptation to new tasks by applying previously acquired knowledge to new challenges~\cite{CF-PA-SL:17,fallah2020provably}. In RL, this evolves into Meta-RL, which trains agents to efficiently adapt to unknown environments~\cite{clavera2018model, AN-IC-SL-RSF-PA-SL-CF:19, fallah2021convergence,toghani2022pars,JB-RV-EZ-ZX-LZ-CF-SW:24}. The Meta-RL framework aims to develop a meta-policy that allows for rapid learning across various tasks, fostering knowledge adaptation to new scenarios.

Model-Agnostic Meta-Learning (MAML), notable for its versatility, focuses on optimizing a model's initial parameters for swift adaptation through minimal adjustments, setting them to be highly responsive to a few policy gradient updates for quick task-specific learning~\cite{CF-PA-SL:17,fallah2021convergence}. On the other hand, the Moreau Envelope (ME) approach \cite{MTT-SL-CAU:23,toghani2023first,CTD-NT-JN:20} introduces a surrogate cost with a regularization term to smooth the optimization process, enhancing the stability of gradient updates and convergence efficiency.

The meta-learning problem for LQR policy design was studied in~\cite{musavi2023convergence,LFT-DZ-JA-HW:24,molybog2021does} and focuses on meta-learning based on the MAML framework. MAML for the LQR problem in a single system and multi-task setting was studied in~\cite{molybog2021does}, which was later extended to a multi-system and multi-task setting by \cite{musavi2023convergence}. In their setup, a finite set of observed tasks is used to design an LQR meta-policy that can effectively and quickly adapt to unobserved tasks, but only local convergence is shown. Later, the authors in~\cite{LFT-DZ-JA-HW:24} studied system heterogeneity and provided global convergence guarantees.

When applying policy gradient techniques, MAML and ME define different surrogate cost functions and access different computational oracles. MAML uses the estimation of the gradient and Hessian of the cost function, whereas ME uses an approximate solution of an inner optimization problem via a first-order oracle. The inner loop accuracy parameter appears in the convergence guarantee, thereby providing the user with explicit control over the quality of the solution. Complied with this theoretical intuition, recent studies have shown that ME allows for a better empirical performance than MAML in multi-task setups \cite{MTT-SL-CAU:23,CTD-NT-JN:20,toghani2023first}.

This work investigates uncertain linear control systems within the LQR framework, applying meta-learning to handle uncertainties. We introduce a meta-policy estimation method using an ME-based framework \cite{toghani2023first} tailored for the LQR in both model-free and model-based settings to facilitate flexible policy initialization for rapid adaptation to new realizations of the uncertain system. Our setup is motivated by the challenges a control system may encounter due to uncertainties rather than having to perform multiple LQR tasks. However, both these scenarios result in a similar underlying framework where the cost incurred by any given policy varies due to changing realizations of the uncertainties or a change in the LQR task. Hence, it is possible to use the approaches presented for both settings interchangeably.

\textit{The main contribution of this paper is a novel first-order meta-RL algorithm (MEMLQR)} that computes LQR meta-policies that can be efficiently adapted for an unseen system realization. By integrating meta-learning principles, our approach improves policy adaptability, ensuring effective performance in various scenarios and quick convergence.

The contributions of this paper are summarized as follows:

\begin{enumerate}
\item We define an augmented cost function where the linear quadratic cost is regularized with Moreau Envelopes to induce personalization amenable to model-based and model-free policy gradient methods for policy design. 
\item We propose a first-order iterative algorithm to optimize the defined augmented cost in a client-server setup.
\item We show the convergence of the proposed algorithm to an approximate first-order stationary point. We also show that the policies generated by the algorithm will incur a finite cost for all the available system realizations.
\item We present a set of numerical results that validate the algorithm's efficiency in minimizing adaptation costs at the testing phase, highlighting the practical benefits of the proposed method. Moreover, we show the proposed method outperforms other approaches based on MAML personalization.
\end{enumerate}

The rest of this article is organized as follows. In Section~\ref{sec:problem_statement}, we formally introduce the problem we are addressing along with the relevant background and propose a solution. In Section~\ref{sec:analysis}, we present the assumptions along with results that indicate the convergence of the proposed algorithm. Section~\ref{sec:simulations} contains numerical examples to indicate the results provided in Section~\ref{sec:analysis}. We conclude our article in Section~\ref{sec:conclusions} and provide the scope for future work.

\noindent \textbf{Notation:} By $\R{}{}$, we denote the set of real numbers, and by $\Zp$, we represent the set of positive integers. For two real numbers $a$ and $b$ such that $b>a$, by $\emph{unif}\pbrack{a,b}$, we define a uniform distribution over the interval $\lcrc{a}{b}$. For any square matrix $M$, $M\succeq 0$ denotes that $M$ is positive semi-definite, and $M\succ 0$ means positive definite. $M$ is Schur if all its eigenvalues lie in the open unit disc.

\section{Problem Formulation and Algorithm}
\label{sec:problem_statement}
In this section, we first introduce the problem setup of finding the optimal policy for a linear system, i.e., single realization. Then, we discuss our approach to finding a meta-policy via personalized costs. Finally, we present our algorithm that minimizes the augmented cost.  

\subsection{Preliminaries}

Consider the family of discrete-time linear time-invariant uncertain dynamical systems 
\begin{equation}
\label{eq:exp_linear}
\dv{}{\iter+1} =  \ematA{} \dv{}{\iter} +  \ematB{} \inpt{}{\iter},  \ \ t=0,1,2,\cdots,
\end{equation}
%
where, $\ematA{}:= \ematA{0}+\sum_{\inode=1}^{\tunA}\unA{\inode}\ematA{\inode}, \ematB{}:= \ematB{0}+\sum_{\jnode=1}^{\tunB} \unB{\jnode}\ematB{\jnode}$, $\dv{}{\iter}\in\R{\dvdim}{}$ is the state and $\inpt{}{\iter}\in\R{\ipdim}{}$ is the input to the system at time $\iter$, $\unA{}:=\pbrack{\unA{1}\;\unA{2}\;\cdots\;\unA{\tunA}}\in\R{\tunA}{}$ and $\unB{}:=\pbrack{\unB{1}\;\unB{2}\;\cdots\;\unB{\tunB}}\in\R{\tunB}{}$ are bounded uncertainties that determine the system dynamics along with matrices $\ematA{0}\in\R{\dvdim\times\dvdim}{}$ and $\ematB{0}\in\R{\dvdim\times\ipdim}{}$, and sets of matrices, $\matAf{}:=\pbrack{\ematA{1},\ematA{2},\cdots,\ematA{\tunA}} \in \pbrack{\R{\dvdim\times\dvdim}{}}^\tunA$ and $\matBf{}:=\pbrack{\ematB{1},\ematB{2},\cdots,\ematB{\tunB}}\in\pbrack{\R{\dvdim\times\ipdim}{}}^\tunB$. It is assumed that all possible realizations of the $\ematA{}-\ematB{}$ pair are controllable. Note that we adopt this model to be true to the earlier literature for uncertain linear systems~\cite{IRP-CVH:86}. Still, we emphasize the fact that a simpler model with uncertain system matrices (such as a norm bound around given matrices) may also be considered.

Moreover, consider a set $\nodes$ of $\tnodes = |\nodes|$ realizations of~\eqref{eq:exp_linear},
\begin{equation}
\label{eq:dynamics}
\dv{\inode}{\iter+1} = \matA{\inode} \dv{\inode}{\iter} + \matB{\inode} \inpt{\inode}{\iter},   \ \ i=1,2,\cdots,V, 
\end{equation}
where $\dv{\inode}{\iter}\in\R{\dvdim}{}$ is the state and $\inpt{\inode}{\iter}\in\R{\ipdim}{}$ is the input for the $\inode^{\thh}$ system at time $\iter$, and $\matA{\inode}\in\R{\dvdim\times\dvdim}{}$ and $\matB{\inode}\in\R{\dvdim\times\ipdim}{}$ are the system matrices, and the initial state $x_i^0$ is randomly drawn from distribution $\mathcal{D}_i$. 

The standard LQR problem for each realization $i\in \nodes$ is
\begin{equation}
\label{eq:cost}
\min_{\{\inpt{\inode}{\iter}\}_{\iter=0}^{\infty}} \mathbb{E}_{\dv{\inode}{0}\sim \mathcal{D}_i} \left[ \sum_{\iter=0}^{\infty}{\dv{\inode}{\iter}}^{\top}\matQ\dv{\inode}{\iter} + {\inpt{\inode}{\iter}}^{\top} \matR{} \inpt{\inode}{\iter} \right],
\end{equation}
subject to~\eqref{eq:dynamics}, where $\matQ\in\R{\dvdim\times\dvdim}{}$ and $\matR{}\in\R{\ipdim\times\ipdim}{}$ are the cost matrices such that $\matQ\succeq 0$ and $\matR{}\succ 0$. Without loss of generality, it is assumed that the cost matrices $\matQ$ and $\matR{}$ are the same for all the systems. This models a scenario where the controllers have the same goal or are trying to accomplish a similar task. The optimal control policy for the above cost function is a linear feedback policy of form $\inpt{\inode}{\iter}=-\mgain{\inode}{}\dv{\inode}{\iter}$ for a policy $\mgain{\inode}{}\in\R{\ipdim\times\dvdim}{}$~\cite[Chapter 3]{bertsekas2012dynamic}. Therefore, the cost can be expressed as
\begin{equation}
\label{eq:cost2}
\kcost{\inode}{\mgain{\inode}{}} := \mathbb{E}_{\dv{\inode}{0}\sim \mathcal{D}_i} \left[ \sum_{\iter=0}^{\infty}{\dv{\inode}{\iter}}^{\top}\big(\matQ + {\mgain{\inode}{}}^{\top}R {\mgain{\inode}{}}\big)\dv{\inode}{\iter} \right].
\end{equation}

Note that, for $\kcost{\inode}{\mgain{\inode}{}}<\infty$ (i.e., $\mgain{\inode}{}$ is stable), we require that $\matA{\inode}+\matB{\inode}\mgain{\inode}{}$ is Schur. 

The seminal work in~\cite{MF-RG-SK-MM:18} established that the cost function $\kcost{\inode}{\mgain{\inode}{}}$ is non-convex with respect to $\mgain{\inode}{}$. However, its gradient domination and local smoothness properties make policy gradient iteration methods converge globally in model-based and model-free settings under the assumption that the initial state distribution is such that the covariance matrix given by $  \EE{\dv{\inode}{0}\sim \mathcal{D}_i} {\dv{\inode}{0}{\dv{\inode}{0}}^{\top}}$ is full rank. 

The gradient descent iterates for a system $\inode \in \nodes$ for finding the optimal policy is stated below:
\begin{equation}
\label{eq:deslqr}
\fgain{\oc+1} = \fgain{\oc} - \olr \gradkcost{\inode}{\fgain{\oc}},
\end{equation}
where, $\fgain{\oc}\in\R{\ipdim\times\dvdim}{}$ is the feedback at $\oc^{\text{th}}$ iteration of the descent, $\olr$ is the step-size and the gradient is given by
\begin{equation}
\label{eq:gradlqr}
\gradkcost{\inode}{\fgain{\oc}} {=} 2\bbrack{\bbrack{\matR{}{+}{\matB{\inode}}^{\top}\matP{\fgain{\oc}}\matB{\inode}}\fgain{\oc}{-}{\matB{\inode}}^{\top}\matP{\fgain{\oc}}\matA{\inode}}\matS{\fgain{\oc}},
\end{equation}
with,
\begin{align}
\matP{\fgain{\oc}} &{=} \matQ {+} {\fgain{\oc}}^{\top}\matR{}\fgain{\oc} {+} \bbrack{\matA{\inode}{-}\matB{\inode}\fgain{\oc}}^{\top}\matP{\fgain{\oc}}\bbrack{\matA{\inode}{-}\matB{\inode}\fgain{\oc}},\\
\matS{\fgain{\oc}} &{=} \mathbb{E}_{\dv{\inode}{0}\sim \mathcal{D}_i} \left[ \sum_{\iter=0}^{\infty} \dv{\inode}{\iter}{\dv{\inode}{\iter}}^{\top} \right].
\end{align}
The iterates in~\eqref{eq:deslqr} need to be initialized at a stable policy $\fgain{0}$ (i.e., $\kcost{\inode}{\fgain{0}} <\infty$). For the case where the model is available, it is easy to calculate the gradient via~\eqref{eq:gradlqr}. However, in the case of a model-free setup, the system and task parameters are unknown. Therefore, the gradient may be estimated by a zeroth-order algorithm~\cite[Algorithm 1]{MF-RG-SK-MM:18}.

\subsection{Moreau Envelope-based LQR Meta-Policy Estimation}

While the approach presented by the policy gradient iterations in~\eqref{eq:deslqr} is efficient for the computation of approximate optimal policies, here we focus on the task of learning good initialization that can be efficiently adapted for the LQR problem on unseen realizations of~\eqref{eq:dynamics}. We propose to use the Moreau Envelope (ME) approach for meta-learning~\cite{CTD-NT-JN:20,toghani2023first,MTT-SL-CAU:23} that defines an optimization problem
\begingroup
\allowdisplaybreaks
\begin{subequations}
\begin{align}
\label{eq:melqr}
&\min_{ \gain{} \in \cap \stable{\inode}}\mcost{\mgain{}} := \sum_{\inode\in\nodes} \mcost{\inode}{\mgain{}},\\
 \text{with} \ \  & \mcost{\inode}{\mgain{}} := \min_{\mgain{\inode}{} \in  \stable{\inode}} \kcost{\inode}{\mgain{\inode}{}} + \frac{\reg}{2} \normb{\mgain{\inode}{}-\mgain{}}^2,  \label{eq:imelqr}
\end{align}
\end{subequations}
\endgroup
where $\lambda \in (0,\infty)$ is a regularization parameter.

\begin{remark}
For any $\inode\in\nodes$, the minimizer $\mgain{\inode}{*}$ for~\eqref{eq:imelqr} will be such that $\matA{\inode}+\matB{\inode}\mgain{\inode}{*}$ is Schur. We define the set of stabilizing policies for any $\inode\in\nodes$ as $\stable{\inode}:=\setdeff{\gain{}\in \R{\ipdim\times\dvdim}{}}{\matA{\inode}+\matB{\inode}\gain{}\text{ is Schur}}$.
\end{remark}

Setting $\lambda=0$ in ME equates the algorithm to local RL, concentrating exclusively on optimizing the current task without leveraging insights from other tasks. Conversely, as $\lambda$ approaches infinity, the algorithm defaults to a naive averaging of costs, disregarding the unique characteristics of each task and resulting in a broadly applicable but non-personalized policy. This is obtained by augmenting the cost defined in~\eqref{eq:cost} with a regularizer term and summing this augmented cost over all the systems in the set $\nodes$. We define a cost based on the Moreau Envelopes inspired by the personalized federated learning setup~\cite{CTD-NT-JN:20, MTT-SL-CAU:23}. 
The meta-policy is given by, $\gain{*} = \argmin_{\gain{} \in \cap \stable{\inode}} \mcost{\mgain{}}$. The minimization of this cost $\mcost{\mgain{}}$ is a bilevel optimization problem where the solutions of both outer and inner minimization problems implicitly depend on each other. Here, the inner problem minimizes the cost induced when the LQR cost is regularized with a proximity term. 

We propose a gradient-based iterative framework termed \emph{MEMLQR}: \underline{M}oreau \underline{E}nvelope based \underline{M}eta \underline{L}inear \underline{Q}uadratic \underline{R}egulator, shown in Algorithm~\ref{algo:dis_com}, to solve the bilevel optimization problem~\eqref{eq:melqr}, for which the gradient is 
%
%
\begingroup
\allowdisplaybreaks
\begin{subequations}
\begin{align}
\label{eq:gmelqr}
&\hspace{17mm}\gradmcost{\mgain{}} = \sum_{\inode\in\nodes} \gradmcost{\inode}{\mgain{}},\\
\label{eq:gimelqr}
&\text{such that, }\gradmcost{\inode}{\mgain{}} =  \lambda\big(\mgain{}- \dmcgain_{\inode} \bbrack{\mgain{}} \big),  \\
\label{eq:inner}
&\text{with} \ \ \dmcgain_{\inode}\bbrack{\mgain{}} = \argmin_{\Breve{\mgain{}}\in \stable{i}}  \kcost{\inode} { \Breve{\mgain{}} } + \frac{\reg}{2} \| \Breve{\mgain{}}-\mgain{}\|^2.
\end{align}
\end{subequations}
\endgroup

\begin{remark}
The computation of~\eqref{eq:gimelqr} requires solving~\eqref{eq:inner}. The gradient for the augmented inner cost function is given by $\gradkcost{\inode}{\Breve{\mgain{}}}
+ \reg \big(\Breve{\mgain{}}-\mgain{} \big)$, 
where, $\gradkcost{\inode}{\Breve{\mgain{}}}$ can be obtained using~\eqref{eq:gradlqr} or the zeroth-order approach~\cite[Algorithm 1]{MF-RG-SK-MM:18}, and $\mgain{}$ is the current estimate of the meta feedback policy. Thus, this approach is equally amenable to model-based or model-free scenarios.
\end{remark}

 As the cost induced by an estimate of the meta-policy at different systems is independent of each other, a client-server model can be used for computation as individual clients need not share their data. Once the inner optimizer is obtained, a gradient descent step is performed locally for the outer problem, and a local estimate for the meta-policy is obtained. Once all the systems have an estimate of the meta-policy, these estimates are communicated to the server. The meta-policy at the server is updated using estimates from different systems and the previous estimate from the server. This iterative gradient descent is performed until the number of outer iterations guarantees the required estimate accuracy.
\begin{algorithm}[!ht]
\caption{\emph{MEMLQR}: \underline{M}oreau \underline{E}nvelope based \underline{M}eta \underline{L}inear \underline{Q}uadratic \underline{R}egulator}
\label{algo:dis_com}
\SetAlgoLined
\DontPrintSemicolon
\SetKwInOut{ini}{Initialize}
\SetKwInOut{giv}{Data}
\giv{Number of outer iterations $\ochorizon$, number of inner iterations $\tinner$, inner step-size $\olr$, outer step-size $\mlr$, accuracy threshold $\biopt$, initial value of the policy $\gain{0}$.}
\BlankLine
\For{$\oc=0,1,\cdots,\ochorizon-1$}{
Communicate the policy at the server, $\mgain{\inode}{\oc,0} \gets \gain{\oc}, \text{ for all }\inode\in\nodes$\;
\For{$\inode\in\nodes$}{
\For{$\ic=0,1,2,\cdots,\tinner-1$}{
Compute $\bar{K}$ s.t.  \; $\|\dmcgain_{\inode}\bbrack{\mgain{\inode}{\oc,\ic}}-\bar{K}\|\leq\biopt$, c.f.~\eqref{eq:inner}  \;
$\mgain{\inode}{\oc,\ic+1}\gets\mgain{\inode}{\oc,\ic}-\olr \reg \pbrack{\mgain{\inode}{\oc,\ic}-\bar{K}}$
}
 $\gain{\oc+1}\gets\pbrack{1-\mlr}\gain{\oc}+\frac{\mlr}{\tnodes}\sum_{\inode\in\nodes}\mgain{\inode}{\oc,\tinner}$
}}
\Return $\gain{\ochorizon}$
\end{algorithm}

\section{Convergence analysis}
\label{sec:analysis}

This section formally states the associated assumptions and auxiliary results for the convergence analysis for Algorithm~\ref{algo:dis_com}. Due to space considerations, the proofs have been moved to~\cite{AA-MTT-CAU:24}. We show in our main result Theorem~\ref{thm:convergence} that, after a threshold number of iterations, the algorithm converges to a neighbourhood of a first-order stationary point of the Moreau regularized cost $\mcost{\cdot}$. This neighbourhood is determined by the number of iterations, the regularizer, the system and cost matrices, and an error incurred in the inner loop calculations. The policy gradient method we use as a baseline for designing our algorithm finds the gradient at the current iterate of the feedback policy by simulating the system under that policy from random initial states. The assumption presented below ensures that all the states are visited with a non-zero probability while exploring.
\begin{assumption}
\longthmtitle{Persistence of excitation-like}
\label{assump:persistence}
Let $\dv{\inode}{0}$ be the initial state of realization $i\in\mathcal{V}$, then $ \mathbb{E}_{\dv{\inode}{0}\sim \mathcal{D}_i} \left[ \dv{\inode}{0}{\dv{\inode}{0}}^{\top} \right] $ is full rank.
\end{assumption}

The above assumption is well-known and standard in the control literature~\cite{MF-RG-SK-MM:18, BG-PME-TS:20}. The system should be stable at the initial guess for any policy gradient algorithm to work. The assumption below ensures that none of the system's available realizations become unstable when we initialize the algorithm using a random initial gain. Such instability will cause the gradient to be unavailable at those realizations. 
\begin{assumption}
\longthmtitle{Bound on initial cost} 
\label{assump:cost}
Let the set $\stable{} := \cap \stable{\inode}$ be non-empty and $\gain{0}\in \stable{}$ be the policy used to initialize Algorithm~\ref{algo:dis_com}, then $C(\gain{0})<\infty$.
\end{assumption}

The following assumption ensures that the systems generated using the available uncertainty realizations are sufficiently close to each other. The bound assumed below is critical for the analysis that follows.
\begin{assumption}
\longthmtitle{Bounded heterogeneity}
\label{assump:diversity}
For all $\mgain{}\in \stable{}$ such that $\kcost{}{\mgain{}}<\infty$, there exists a constant $\bgrad>0$ satisfying,
\begin{equation*}
\frac{1}{\tnodes}\sum_{\inode\in\nodes}\normb{\gradkcost{\inode}{\mgain{}} -\gradkcost{}{\mgain{}}}^2 \leq \bgrad^2.
\end{equation*}
\end{assumption}

The motivation for this assumption is similar to that of~\cite[Lemma 4]{LFT-DZ-JA-HW:24}, where the authors had established a similar bound on the diversity of gradients based on the diversity in the system and cost matrices. In our case, the diversity in system matrices is bounded as the uncertainties $\unA{}$ and $\unB{}$ are bounded. Therefore, Assumption~\ref{assump:diversity} aligns with the established theory. Algorithm~\ref{algo:dis_com} requires the estimation of an optimizer of the inner problem~\eqref{eq:inner} up to a desired accuracy. The LQR cost as a function of the feedback policy satisfies the PL-inequality~\cite[Lemma 3]{MF-RG-SK-MM:18}, i.e., let $\mgain{\inode}{*}$ be the optimal policy for a system realization $\inode\in\nodes$, then for any stable policy $\mgain{}$ and some constant $\graddom>0$, it holds that
\begin{equation}
\label{eq:graddom}
\kcost{\inode}{\mgain{}}-\kcost{\inode}{\mgain{\inode}{*}} \leq \frac{\graddom}{2} \norm{\gradkcost{\inode}{\mgain{}}}^2,
\end{equation}
In our approach, we add a quadratic regularizer term to this gradient-dominated function. We have established in Lemma~\ref{lem:megraddom} that the regularized cost~\eqref{eq:imelqr} also satisfies the PL inequality. This means it is possible to solve the inner minimization problem to any specified accuracy $\biopt$ using gradient descent (or Gauss-Newton or Natural policy gradient) iterations of order $\bigO{\log{(1/\biopt)}}$~\cite[Theorem 7, Theorem 9]{MF-RG-SK-MM:18}. Next, we build a sequence of auxiliary lemmas that will be useful in analyzing Algorithm~\ref{algo:dis_com}. First, we show the gradient dominance property of the Moreau regularized cost in~\eqref{eq:imelqr}.
\begin{lemma}
\longthmtitle{Gradient dominance of Moreau Envelope cost}
\label{lem:megraddom}
Consider a policy $\Breve{\mgain{}}\in\R{\ipdim\times\dvdim}{}$, then
\begin{equation}
\label{eq:MoreauPL}
\hspace{-0.2cm}\mcost{\inode}{\Breve{\mgain{}}} {-} \min_{\mgain{}\in \stable{i}} \mcost{\inode}{\mgain{}} \leq  \pbrack{\frac{\graddom}{2}{+}\frac{1}{2\reg}} \normb{\gradmcost{\inode}{\Breve{\mgain{}}}}^2.
\end{equation}
\end{lemma}
\excludethis{
\begin{proof}
We extend the proof of~\cite[Proposition 4.1]{VA-NG-SV:22} to functions that satisfy PL inequality. For the policy $\Breve{\mgain{}}$ we have,  
\begin{align*}
\mcost{\inode}{\Breve{\mgain{}}} {-} \min_{\mgain{}\in \stable{i}} \mcost{\inode}{\mgain{}} {=} \kcost{\inode}{\widehat{\mgain{}}} {+} \frac{\reg}{2} \normb{\widehat{\mgain{}}-\Breve{\mgain{}}}^2 {-} \kcost{\inode}{\mgain{\inode}{*}},
\end{align*}
where, $\widehat{\mgain{}}=\argmin_{\mgain{}} \kcost{\inode}{\mgain{}} + \frac{\reg}{2} \normb{\mgain{}-\Breve{\mgain{}}}^2$ and $\mgain{\inode}{*}=\argmin_{\mgain{}}\kcost{\inode}{\gain{}}$. The equality follows because $\cost{\inode}$ satisfies the PL inequality, therefore $\min_{\mgain{}} \mcost{\inode}{\mgain{}}=\kcost{\inode}{\mgain{\inode}{*}}$. Now using~\eqref{eq:graddom} we have,
\begin{align*}
\mcost{\inode}{\Breve{\mgain{}}} {-} \min_{\mgain{}\in \stable{i}} \mcost{\inode}{\mgain{}} &\leq \frac{\graddom}{2} \normb{\gradkcost{\inode}{\widehat{\mgain{}}}}^2 + \frac{\reg}{2} \normb{\widehat{\mgain{}}-\Breve{\mgain{}}}^2 \\
&\overset{(a)}{=} \frac{\graddom}{2} \normb{\reg\bbrack{\Breve{\mgain{}}-\widehat{\mgain{}}}}^2 + \frac{\reg}{2} \normb{\widehat{\mgain{}}-\Breve{\mgain{}}}^2 \\
&\overset{(b)}{=}\pbrack{\frac{\graddom}{2}+\frac{1}{2\reg}} \normb{\gradmcost{\inode}{\Breve{\mgain{}}}}^2,
\end{align*}
where (a) follows from the optimality of $\widehat{\mgain{}}$ and the fact that $\kcost{\inode}{\mgain{\inode}{*}}\geq 0$, and (b) follows from the definition of the gradient of Moreau regularized cost. 
\end{proof}
}
The following lemma shows how the cost function's smoothness property is inherited from the Moreau regularized cost.
\begin{lemma}
\longthmtitle{Local smoothness of cost functions}
\label{lem:smooth}
For any given system realization $\inode\in\nodes$ and (stable) policies $\mgain{1}{},\mgain{2}{}\in  \stable{\inode}$ such that $\normb{\mgain{1}{}-\mgain{2}{}}\leq\gaindiv$, there exists $\smooth>0$ such that,
\begin{equation}
\label{eq:costsmooth}
\normb{\gradkcost{\inode}{\mgain{1}{}}-\gradkcost{\inode}{\mgain{2}{}}}\leq \smooth \normb{\mgain{1}{}-\mgain{2}{}}
\end{equation}
and, for some constant $\kappa>1$ if $\reg>\kappa \smooth$ it follows that,
\begin{equation}
\label{eq:mecostsmooth}
\normb{\gradmcost{\inode}{\mgain{1}{}}-\gradmcost{\inode}{\mgain{2}{}}}\leq \smoothME \normb{\mgain{1}{}-\mgain{2}{}},
\end{equation}
where, $\smoothME:={\smooth}/{(\kappa-1)}$.
\end{lemma}
\excludethis{
\begin{proof}
Inequality~\eqref{eq:costsmooth} follows from~\cite[Lemma 25]{MF-RG-SK-MM:18}, where it was shown that for $\gaindiv:=\min\pbrack{\frac{\sigma_{\min}\bbrack{\matQ{\inode}}}{4\kcost{\inode}{\mgain{2}{}}\normb{\matB{\inode}}\bbrack{\normb{\matB{\inode}+\matB{\inode}\mgain{2}{}}+1}},\normb{\mgain{2}{}}}$, $\smooth$ is a polynomial in $\frac{\cost{\inode}\pbrack{\gain{0}}}{\sigma_{\min}\bbrack{\EE{\dv{\inode}{0}\sim\istate{\inode}}{\dv{\inode}{0}{\dv{\inode}{0}}^{\top}}}\sigma_{\min}\pbrack{\matQ}}$, $\EE{}{\normb{\dv{\inode}{0}}^2}$, $\normb{\matA{\inode}}$, $\normb{\matB{\inode}}$, $\normb{\matR{}}$, and ${1}/{\sigma_{\min}\pbrack{\matR{}}}$. The second inequality~\eqref{eq:mecostsmooth} follows from~\cite[Lemma 4.5]{MTT-SL-CAU:23}, where it was shown that the smoothness properties get transferred to the Moreau envelope regularized cost function.     
\end{proof}}

We have assumed a diversity (c.f. heterogeneity) bound for the gradients of LQR costs in Assumption~\ref{assump:diversity}. Still, our approach minimizes a cost function, the sum of Moreau regularized LQR costs. The following Lemma quantifies the heterogeneity of the Moreau regularized costs among various system realizations under Assumption~\ref{assump:diversity} and in the view of Lemma~\ref{lem:smooth}. 
\begin{lemma}
\longthmtitle{\cite[Lemma 2]{CTD-NT-JN:20}, Bounded Heterogeneity of Moreau Envelope costs}
\label{lem:diversity}
For some policy $\mgain{}\in\R{\ipdim\times\dvdim}{}$, such that $\kcost{}{\mgain{}}<\infty$ and $\reg>2\sqrt{2}\smooth$, it follows that
\begin{equation}
\begin{split}
&\frac{1}{\tnodes}\sum_{\inode\in\nodes}\normb{\gradmcost{\inode}{\mgain{}}-\gradmcost{\mgain{}}}^2 \\
&\hspace{1.3cm}\leq \frac{8\smooth^2}{\reg^2-8\smooth^2} \normb{\gradmcost{\mgain{}}} + \frac{2\reg^2}{\reg^2-8\smooth^2} \bgrad^2.  
\end{split}
\end{equation}
\end{lemma}

The following lemma quantifies the drift induced due to local iterations $P$ in Algorithm~\ref{algo:dis_com}. The estimate of the policy at $\oc+1$ outer iteration can be written as
\begin{align}
\label{eq:outerdescent}
\gain{\oc+1} = \gain{\oc} - \vlr \gradgk{\oc}
\end{align}
where, $\vlr := \olr \mlr \tinner$ and $\gradgk{\oc} := \frac{1}{\tnodes\tinner} \sum_{\inode\in\nodes}\sum_{\ic=0}^{\tinner} \gradgm{\inode}{\oc}{\ic}$, such that $\gradgm{\inode}{\oc}{\ic}:=\reg\big(\mgain{\inode}{\oc,\ic}-\mcgain{\inode}{\oc}{\ic}\big)$.
\begin{lemma}
\longthmtitle{~\cite[Lemma 5]{CTD-NT-JN:20}, Cumulative gradient drift}
\label{lem:drift}
For algorithm~\ref{algo:dis_com}, at an outer iteration $\oc$ and under the assumption that $\vlr\leq {\mlr}/({2\smoothME})$, the following holds:
\begin{equation}
\begin{split}
&\frac{1}{\tnodes\tinner} \sum_{\inode\in\nodes}\sum_{\ic=0}^{\tinner}\normb{\gradgm{\oc}{\inode}{\ic}-\gradmcost{\inode}{\gain{\oc}}}^2 \leq 2\reg^2\biopt^2 +\\
&\hspace{1cm} \frac{16\smoothME\vlr^2}{\mlr^2}\pbrack{\frac{2\reg^2\biopt^2}{\tinner} + 3\frac{1}{\tnodes} \sum_{\inode\in\nodes}\normb{\gradmcost{\inode}{\gain{\oc}}}^2}.    
\end{split}
\end{equation}
\end{lemma}
The following proposition shows that Moreau regularized cost for a realization $\inode\in\nodes$ decreases during the iterations of the inner loop.
\begin{proposition}
\longthmtitle{Decrease in the Moreau Envelope cost}
\label{prop:locstability}
Let Assumptions~\ref{assump:persistence},~\ref{assump:cost} and~\ref{assump:diversity} hold, and $\olr\leq{1}/{\smoothME}$ with $\tinner\geq 1$. Consider an iterate $\gain{\oc}$ inside Algorithm~\ref{algo:dis_com}, such that $\kcost{\inode}{\gain{\oc}}<\infty$, then $\mcost{\inode}{\mgain{\inode}{\oc,\tinner}}\leq\mcost{\inode}{\gain{\oc}}$ and $\kcost{\inode}{\mgain{\inode}{\oc,\tinner}}<\infty$.
\end{proposition}
\excludethis{
\begin{proof}
By Step 2 of the MEMLQR algorithm, it can be seen that $\mgain{\inode}{\oc,0}=\gain{\oc}$. Then, using the smoothness of the Moreau regularized cost, we get,
\begin{align*}
 &\mcost{\inode}{\mgain{\inode}{\oc,1}} - \mcost{\inode}{\mgain{\inode}{\oc,0}} \\
 &\hspace{0.4cm}\overset{(a)}{\leq} \inprodb{\gradmcost{\inode}{\mgain{\inode}{\oc,0}}}{\mgain{\inode}{\oc,1}-\mgain{\inode}{\oc,0}} + \frac{\smoothME}{2} \normb{\mgain{\inode}{\oc,1}-\mgain{\inode}{\oc,0}}^2 \\
 &\hspace{0.4cm}\overset{(b)}{=} -\olr\inprodb{\gradmcost{\inode}{\mgain{\inode}{\oc,0}}}{\gradmcost{\inode}{\mgain{\inode}{\oc,0}}} \\
 &\hspace{4cm}+ \frac{\smoothME\olr^2}{2} \normb{\gradmcost{\inode}{\mgain{\inode}{\oc,0}}}^2 \\
 &\hspace{0.4cm}= \pbrack{\frac{\smoothME\olr^2}{2}-\olr} \normb{\gradmcost{\inode}{\mgain{\inode}{\oc,0}}}^2 \\
 &\hspace{0.4cm}\overset{(c)}{\leq} -\frac{\olr}{2} \normb{\gradmcost{\inode}{\mgain{\inode}{\oc,0}}}^2 \\
 &\hspace{0.4cm}\overset{(d)}{\leq} -\frac{\olr}{2} \pbrack{\frac{2\reg}{\graddom\reg+1}} \pbrack{\mcost{\inode}{\mgain{\inode}{\oc,0}} - \mcost{\inode}{\mgain{\reg,\inode}{*}}},
\end{align*}
where (a) follows from Lemma~\ref{lem:smooth}, (b) is due to Step 6 of the MEMLQR algorithm, for (c) we assume that $\olr\leq {1}/{\smoothME}$ and (d) is the result of Remark~\ref{lem:megraddom}. By adding $\mcost{\inode}{\mgain{\inode}{0,0}} - \mcost{\inode}{\mgain{\reg,\inode}{*}}$ to both sides, we get,
\begin{align*}
\mcost{\inode}{\mgain{\inode}{\oc,1}} &- \mcost{\inode}{\mgain{\reg,\inode}{*}} \\
&\leq  \pbrack{1-\frac{\olr\reg}{\graddom\reg+1}} \pbrack{\mcost{\inode}{\mgain{\inode}{\oc,0}} - \mcost{\inode}{\mgain{\reg,\inode}{*}}}
\end{align*}
implying that, $\mcost{\inode}{\mgain{\inode}{\oc,1}}\leq\mcost{\inode}{\mgain{\inode}{\oc,0}}$. Now, by telescoping the above arguments over all the iterates in the inner loop, we obtain $\mcost{\inode}{\mgain{\inode}{\oc,\tinner}}\leq\mcost{\inode}{\gain{\oc}}$.

Now, for an iterate $\mgain{\inode}{\oc,\ic}$, from the Step 5 of the MEMLQR algorithm, we know that,
\begin{align*}
\mcost{\inode}{\mgain{\inode}{\oc,\ic}} = \kcost{\inode}{\mcgain{\inode}{\oc}{\ic}} + \frac{\reg}{2} \normb{\mgain{\inode}{\oc,\ic}-\mcgain{\inode}{\oc}{\ic}}^2,
\end{align*}
which implies that, $\kcost{\inode}{\mcgain{\inode}{\oc}{\ic}}\leq \kcost{\inode}{\mgain{\inode}{\oc,\ic}}$.
Now, by Step 6 in the MEMLQR algorithm,
\begin{align*}
\mgain{\inode}{\oc,\ic+1}=\pbrack{1-\olr \reg}\mgain{\inode}{\oc,\ic}+\olr \reg \mcgain{\inode}{\oc}{\ic},
\end{align*}
therefore, we get
\begin{align*}
 \kcost{\inode}{\mcgain{\inode}{\oc}{\ic}}\leq \kcost{\inode}{\mgain{\inode}{\oc,\ic+1}} \leq  \kcost{\inode}{\mgain{\inode}{\oc,\ic}}.   
\end{align*}
Finally, extending the above arguments to all the iterates in the inner loop, we get $\kcost{\inode}{\mgain{\inode}{\oc,\tinner}}\leq\kcost{\inode}{\mgain{\inode}{\oc,0}}<\infty$.
\end{proof}}

Next, we present a corollary that establishes that the policy generated by local iterations for a system realization, when used for another system realization, incurs a finite cost.
\begin{corollary}
\longthmtitle{Finite cost by local policies}
\label{cor:stable}
Let the conditions in Proposition~\ref{prop:locstability} be satisfied. Then, consider the local estimate $\mgain{\inode}{\oc,\tinner}$ at a system realization $\inode\in\nodes$ for an outer iteration $\oc$ of Algorithm~\ref{algo:dis_com}, it follows that $\kcost{\jnode}{\mgain{\inode}{\oc,\tinner}}<\infty$ for all $\jnode\in\nodes$.
\end{corollary}
\excludethis{
\begin{proof}
By Proposition~\ref{prop:locstability}, $\kcost{\inode}{\mgain{\inode}{\oc,\tinner}}<\infty$ and the diversity in the gradients for different realizations was bounded in Assumption~\ref{assump:diversity}. By combining these statements with \eqref{eq:graddom}, we can infer that $\kcost{\jnode}{\mgain{\inode}{\oc,\tinner}}<\infty$.
\end{proof}}

Algorithm~\ref{algo:dis_com} depends on the gradient calculation at the current iterate, which requires the cost incurred at the iterate to be finite for all the available system realizations. The following result establishes this for the policies Algorithm~\ref{algo:dis_com} generates at every outer iteration.
\begin{corollary}
\longthmtitle{Finite cost at all iterations}
\label{cor:allstable}
Let the conditions in Proposition~\ref{prop:locstability} be satisfied. Then, for any iterate $\gain{\oc}$ generated by Algorithm~\ref{algo:dis_com}, $\kcost{\inode}{\gain{\oc}}<\infty$ for all $\inode\in\nodes$.
\end{corollary}
\excludethis{
\begin{proof}
The statement follows from Assumption~\ref{assump:cost}, \eqref{eq:graddom}, and Corollary~\ref{cor:stable}.
\end{proof}}

\begin{remark}
    Proposition~\ref{prop:locstability}, Corollary~\ref{cor:stable}, and Corollary~\ref{cor:allstable} are analogous to the stability results in \cite[Theorem 1, Theorem 2]{LFT-DZ-JA-HW:24}. Through these, we show that the cost incurred by policies generated inside Algorithm~\ref{algo:dis_com} remains finite for all available system realizations.
\end{remark}

The next theorem presents our main result; it shows the convergence of the iterates produced by Algorithm~\ref{algo:dis_com} to an approximate first-order stationary point of~\eqref{eq:melqr}. 
\begin{theorem}
\longthmtitle{Convergence to a first-order stationary point}
\label{thm:convergence}
Let Assumptions~\ref{assump:persistence},~\ref{assump:cost} and ~\ref{assump:diversity} hold, $\smooth$ and $\smoothME$ from Lemma~7 with $\reg >2\sqrt{2}\smooth$, $\ochorizon\geq 16\smoothME^2\pbrack{1+72\reg^2}^2$,  $\tinner\geq 1$, $\olr\leq{1}/({2\smoothME\tinner})$, $\mlr\geq\max\left\{0.5,{2\smoothME}/{\sqrt{\ochorizon}}\right\}$, and $\vlr={1}/{\sqrt{\ochorizon}}$. Then, the sequence $\left\{\gain{\oc}\right\}_{\oc=0}^{\ochorizon}$ generated by Algorithm~\ref{algo:dis_com} has the following property:
\begin{equation}
\frac{1}{4\ochorizon}\sum_{\oc=0}^{\ochorizon}\norm{\gradmcost{\gain{\oc}}}^2 \leq \frac{1}{\sqrt{\ochorizon}} \big(\mcost{\gain{0}} -\mcost{\gain{*}}\big) {+} \frac{\cons}{\ochorizon}  + \coni   
\end{equation}
where, $\coni = 2\reg^2\biopt^2$ and $\cons = 4\pbrack{\frac{2\reg^2\biopt^2}{\tinner}+\frac{6\reg^2}{\reg^2-8\smooth^2} \bgrad^2}$.
\end{theorem}
\excludethis{
\begin{proof}
Using the $\smoothME$-smoothness of the function $C_\lambda$,
\begin{align}
\label{eq:proof1}
&\mcost{\gain{\oc+1}} - \mcost{\gain{\oc}}\nonumber\\
&\hspace{0.5cm}\leq \inpr{\gradmcost{\gain{\oc}}}{\gain{\oc+1}-\gain{\oc}} + \frac{\smoothME}{2}\normb{\gain{\oc+1}-\gain{\oc}}^2 \nonumber\\
&\hspace{0.5cm}\overset{(a)}{=}  \underbrace{-\vlr\inpr{\gradmcost{\gain{\oc}}}{\gradgk{\oc}}}_{A} + \underbrace{\frac{\vlr^2\smoothME}{2} \normb{\gradgk{\oc}}^2}_{B}, 
\end{align}
where, (a) follows from~\eqref{eq:outerdescent}. We bound the term A in the above summation as,
\begin{align}
\label{eq:proof2}
&\inpr{\gradmcost{\gain{\oc}}}{\gradgk{\oc}} \nonumber \\
&\overset{(b)}{=} \normb{\gradmcost{\gain{\oc}}}^2 + \inpr{\gradmcost{\gain{\oc}}}{\gradgk{\oc}-\gradmcost{\gain{\oc}}} \nonumber\\
&\overset{(c)}{\geq} \normb{\gradmcost{\gain{\oc}}}^2 - \frac{1}{2}\normb{\gradmcost{\gain{\oc}}}^2 \nonumber \\
&\hspace{1cm}-\frac{1}{2}\normb{\gradgk{\oc} -\gradmcost{\gain{\oc}}}^2 \nonumber \\
&\overset{(d)}{=}\frac{1}{2}\normb{\gradmcost{\gain{\oc}}}^2\nonumber \\
&\hspace{1cm}-\frac{1}{2}\norm{\frac{1}{\tnodes\tinner}\sum_{\inode\in\nodes}\sum_{\ic=0}^{\tinner-1}\pbrack{\gradgm{\inode}{\oc}{\ic}-\gradmcost{\inode}{\gain{\oc}}}}^2
\end{align}
where, (b) and (c) follow from the triangle inequality and the properties of inner product, and (d) follows from~\eqref{eq:outerdescent} and definition of $\gradmcost{\inode}{\gain{\oc}}$. For the term B, we have
\begin{align}
\label{eq:proof3}
\normb{\gradgk{\oc}}^2 &= \normb{\gradgk{\oc}-\gradmcost{\gain{\oc}}+\gradmcost{\gain{\oc}}}^2 \nonumber \\
&\overset{(e)}{\leq} 2\normb{\gradgk{\oc}-\gradmcost{\gain{\oc}}}^2+2\normb{\gradmcost{\gain{\oc}}}^2 \nonumber \\
&\overset{(f)}{=} 2\norm{\frac{1}{\tnodes\tinner}\sum_{\inode\in\nodes}\sum_{\ic=0}^{\tinner-1}\pbrack{\gradgm{\inode}{\oc}{\ic}-\gradmcost{\inode}{\gain{\oc}}}}^2\nonumber\\
&\hspace{3.5cm}+2\normb{\gradmcost{\gain{\oc}}}^2
\end{align}
where, (e) follows from the Jensen's inequality and (f) follows from the relevant definitions. Using the bounds in~\eqref{eq:proof2},~\eqref{eq:proof3} and Jensen's inequality in~\eqref{eq:proof1} we obtain,
\begin{align*}
& \mcost{\gain{\oc+1}} - \mcost{\gain{\oc}} \\
&\hspace{0.5cm}\leq \frac{-\vlr(1-2\vlr\smoothME)}{2}\normb{\gradmcost{\gain{\oc}}}^2 \nonumber \\
&\hspace{1cm}+ \vlr\frac{1+2\vlr\smoothME}{2}\frac{1}{\tnodes\tinner} \sum_{\inode\in\nodes}\sum_{\ic=0}^{\tinner}\normb{\gradgm{\inode}{\oc}{\ic}-\gradmcost{\inode}{\gain{\oc}}}^2 \\
&\hspace{0.5cm}\overset{(g)}{\leq} \frac{-\vlr(1-2\vlr\smoothME)}{2}\normb{\gradmcost{\gain{\oc}}}^2 + \vlr\frac{1+2\vlr\smoothME}{2} 2\reg^2\biopt^2 \\
&\hspace{1mm}+ \vlr\frac{1+2\vlr\smoothME}{2} \frac{16\smoothME^2\vlr^2}{\mlr^2}\pbrack{\frac{2\reg^2\biopt^2}{\tinner} + \frac{3}{\tnodes} \sum_{\inode\in\nodes}\normb{\gradmcost{\inode}{\gain{\oc}}}^2},
\end{align*}
where, (g) is due to Lemma~\ref{lem:drift} where it was imposed that $\vlr^2\leq {\mlr^2}/({4\smoothME^2})$ implying $2\olr^2\smoothME^2\leq {1}/({2\tinner^2})$ as $\vlr=\olr\mlr\tinner$. Now, 
\begin{align*}
& \mcost{\gain{\oc+1}} - \mcost{\gain{\oc}} \\
&\overset{(h)}{\leq} -\vlr\frac{1-2\vlr\smoothME}{2}\normb{\gradmcost{\gain{\oc}}}^2 + \vlr\frac{1+2\vlr\smoothME}{2} 2\reg^2\biopt^2 \\
&\hspace{1cm}+ \vlr\frac{1+2\vlr\smoothME}{2} \frac{16\smoothME^2\vlr^2}{\mlr^2}\left(\frac{2\reg^2\biopt^2}{\tinner} +3\normb{\gradmcost{\gain{\oc}}}^2 \right.\\
&\hspace{2cm}+ \left.\frac{3}{\tnodes} \sum_{\inode\in\nodes}\normb{\gradmcost{\inode}{\gain{\oc}}-\gradmcost{\gain{\oc}}}^2\right) \\
&\overset{(i)}{\leq} -\vlr\frac{1-2\vlr\smoothME}{2}\normb{\gradmcost{\gain{\oc}}}^2 + \vlr\frac{1+2\vlr\smoothME}{2} 2\reg^2\biopt^2 \\
&\hspace{1cm}+\vlr \frac{1+2\vlr\smoothME}{2} \frac{16\smoothME^2\vlr^2}{\mlr^2}\left( \frac{2\reg^2\biopt^2}{\tinner} +\frac{6\reg^2}{\reg^2-8\smooth^2} \bgrad^2\right. \\
&\hspace{2cm}\left.+ \frac{3\reg^2}{\reg^2-8\smooth^2} \normb{\gradmcost{\gain{\oc}}}^2 \right) \\
&= \vlr \frac{1+2\vlr\smoothME}{2} \pbrack{2\reg^2\biopt^2 + \frac{16\smoothME^2\vlr^2}{\mlr^2}\pbrack{\frac{2\reg^2\biopt^2}{\tinner}+\frac{6\reg^2\bgrad^2}{\reg^2-8\smooth^2}}} \\
&\hspace{0.8cm}+ \frac{-\vlr}{2}\underbrace{\pbrack{1-2\vlr\smoothME - \pbrack{1+2\vlr\smoothME} \frac{16\smoothME^2\vlr^2}{\mlr^2} \frac{3\reg^2}{\reg^2-8\smooth^2}}}_{C}\\
&\hspace{6.5cm}. \normb{\gradmcost{\gain{\oc}}}^2,
\end{align*}
where, (h) is due to triangle inequality and (i) follows from Lemma~\ref{lem:diversity}. Now consider the term C above,
\begin{align*}
1-2\vlr\smoothME - &\pbrack{1+2\vlr\smoothME} \frac{16\smoothME^2\vlr^2}{\mlr^2} \frac{3\reg^2}{\reg^2-8\smooth^2} \\
&\overset{(j)}{\geq} 1-2\vlr\smoothME - \pbrack{1+2\vlr\smoothME} \frac{16\smoothME^2\vlr^2}{\mlr^2} 3\reg^2,\\
&\overset{(k)}{\geq} 1-2\vlr\smoothME\pbrack{1+72\reg^2}
\overset{(l)}{\geq} \frac{1}{2},
\end{align*}
where, (j) follows by assuming $\reg^2-8\smooth^2\geq 1$, (k) follows because from Lemma~\ref{lem:drift} we know that $\vlr\leq{\mlr}/({2\smoothME})$ and by assuming $\mlr\geq 0.5$, we have $1+2\vlr\smoothME\leq 1+\mlr \leq 3 \mlr$, and (l) follows by assuming $\vlr\leq {1}/({4\smoothME\pbrack{1+72\reg^2}})$. Then,
\begin{align*}
\mcost{\gain{\oc+1}} - &\mcost{\gain{\oc}} \\
&\leq 
-\frac{\vlr}{4}\normb{\gradmcost{\gain{\oc}}}^2 + \vlr 2\reg^2\biopt^2 \\ 
&\hspace{1.5cm}+ 4 \vlr^3 \pbrack{\frac{2\reg^2\biopt^2}{\tinner} +\frac{6\reg^2}{\reg^2-8\smooth^2} \bgrad^2} \\
&= -\frac{\vlr}{4}\normb{\gradmcost{\gain{\oc}}}^2 + \vlr^3 \cons + \vlr \coni, 
\end{align*}
with the constants defined as $\coni:=2\reg^2\biopt^2$ and $\cons:=4 \pbrack{\frac{2\reg^2\biopt^2}{\tinner}+\frac{6\reg^2}{\reg^2-8\smooth^2} \bgrad^2}$. Now, by rearranging and summing the above over all  $\oc\in\{0,1,\cdots,\ochorizon-1\}$, we get the following on taking the average,
\begin{align*}
 \frac{1}{4\ochorizon}\sum_{\oc=0}^{\ochorizon-1}\normb{\gradmcost{\gain{\oc}}}^2 &\leq \frac{1}{\vlr \ochorizon} \pbrack{\mcost{\gain{0}} - \mcost{\gain{\ochorizon}}} \\
 &\hspace{3cm}+ \vlr^2 \cons + \coni.
\end{align*}
Using the fact that $\mcost{\gain{\ochorizon}}\geq \mcost{\gain{*}}$ and by setting $\vlr = {1}/{\sqrt{\ochorizon}}$, for $\ochorizon\geq 16\smoothME^2\pbrack{1+72\reg^2}^2$, the bound is achieved.
\end{proof}}

The following corollary shows the iteration complexity of Algorithm~\ref{algo:dis_com} by quantifying the number of outer iterations $\ochorizon$ and the required inner accuracy $\biopt$ to reach a $\epsilon>0$ accuracy in the approximate first-order stationary point.
\begin{corollary}
\label{cor:convergence}
Let Assumptions~\ref{assump:persistence},~\ref{assump:cost}, and~\ref{assump:diversity} hold, and the parameters of Algorithm~\ref{algo:dis_com} are set as stated in Theorem~\ref{thm:convergence}.  Moreover, let $0<\epsilon\leq {1}/({256\smoothME^4(1+72\reg^2)^4})$, total number of iterations $\ochorizon = \bigO{{1}/{\epsilon^2}}$ and the required inner accuracy $\biopt = \bigO{\sqrt{\epsilon}}$. Then, the output by Algorithm~\ref{algo:dis_com} is an $\epsilon-$optimal first-order stationary point of~\eqref{eq:melqr}.
\end{corollary}

Corollary~\ref{cor:convergence} shows that as the required accuracy increases (i.e., $\epsilon$ decreases), there is a quadratic increase in the number of outer iterations required, and the required inner accuracy also increases as~$\sqrt{\epsilon}$. Note that the convergence result in \cite{LFT-DZ-JA-HW:24} is up to a non-vanishing constant ball around the optimal solution (\cite[Theorem~3, Theorem~4]{LFT-DZ-JA-HW:24}) while the precision of our method is arbitrary by the precision of the inner optimization problem.

\section{Numerical experiments}
\label{sec:simulations}
In this section, we present numerical experiments to illustrate the performance of the proposed MEMLQR algorithm. To this end, we divide our experiments into two parts; the first introduces a model for uncertain linear systems. We use this model to study the convergence properties of the MEMLQR algorithm and demonstrate adaptation to a system realization. In the second part, we compare our setup to that of~\cite{musavi2023convergence} and~\cite{LFT-DZ-JA-HW:24}. We also present the results of a numerical experiment conducted based on the example provided in~\cite{LFT-DZ-JA-HW:24}.

\subsection{Uncertain linear system}
\label{ssec:numerical}

We used the following model parameters for the experiments in~\eqref{eq:exp_linear}. The uncertainties in the system are $\unA{}\sim \unif{-1,1}^2$ and $\unB{}\sim\unif{-1,1}^2$, where $\unif{-1,1}$ represents a uniform distribution between $-1$ and $1$. The matrices generating the system are $\small\ematA{0}= \begin{bmatrix}
0.7 & -0.3 & 0.0 & 0.1 \\
0.5 & -0.4 & 0.3 & 0.0\\
0.0 & 0.4 & 0.2 & -0.1\\
0.2 & 0.0 & 0.4 & 0.6
\end{bmatrix}$, $\ematA{1}$ and $\ematA{2}$ are lower and upper triangular matrices, respectively, with all entries as 0.1 and
\begin{align*}
\small
\ematB{0} = \begin{bmatrix}
0.3 & 0.2 \\ 
0.1 & 0.5 \\ 
0.4 & 0.1 \\ 
0.0 & 0.1   
\end{bmatrix}, \ematB{1} = \begin{bmatrix}
0.1 & 0.0 \\ 
0.0 & 0.1 \\ 
0.1 & 0.0 \\ 
0.0 & 0.1   
\end{bmatrix},  \ematB{2} = \begin{bmatrix}
0.0 & 0.1 \\ 
0.1 & 0.0 \\ 
0.0 & 0.1 \\ 
0.1 & 0.0   
\end{bmatrix}.
\end{align*}

The initial state is given by $\dv{}{0}\sim\pbrack{\text{\emph{unif}}\pbrack{-10,10}}^4$ and the task is specified by a cost of form~\eqref{eq:cost} where the cost matrices are given by $\matQ = \text{\emph{diag}}\pbrack{1,2,3,4}$ and $\matR{} = \text{\emph{diag}}\pbrack{1,2}$.

Four system realizations were generated using the parameters provided above. The MEMLQR algorithm was implemented with $\ochorizon=300$, $\tinner=2$, $\olr = 0.1$, $\mlr=1$ and $\reg = 0.2$. To solve the optimization problem in Step 5, we use the model-based gradient descent algorithm from~\cite{MF-RG-SK-MM:18}. We show the convergence characteristics of the algorithm by plotting the Moreau regularized cost $(\mcost{K^s})$ against the number of outer iterations $(s)$ in Fig.~\ref{fig:convergence}. Note that we slightly abuse our notations here by representing the cost incurred while adapting to a new system realization $z$ as $C_{z}$. In Fig.~\ref{fig:accuracy} we plot the accuracy $\pbrack{1-{|C_{z}(K^N)-C_{z}(K^n)|}/{C_{z}(K^N)}}$ for the cost incurred while adapting to a new system, by initializing using the policy that minimizes the cumulative cost and the MEMLQR policy generated previously. It can be observed that the MEMLQR policy adapts faster to the optimal policy of the unseen realization $z$. We also plot in Fig.~\ref{fig:trajectories} the state and input trajectories of the unseen system realization $z$ for the policy generated when a policy gradient algorithm is initialized at the MEMLQR policy. 
\begin{figure*}[t!]
\centering
\begin{subfigure}[t]{0.32\textwidth}
\centering
\includegraphics[scale=0.23]{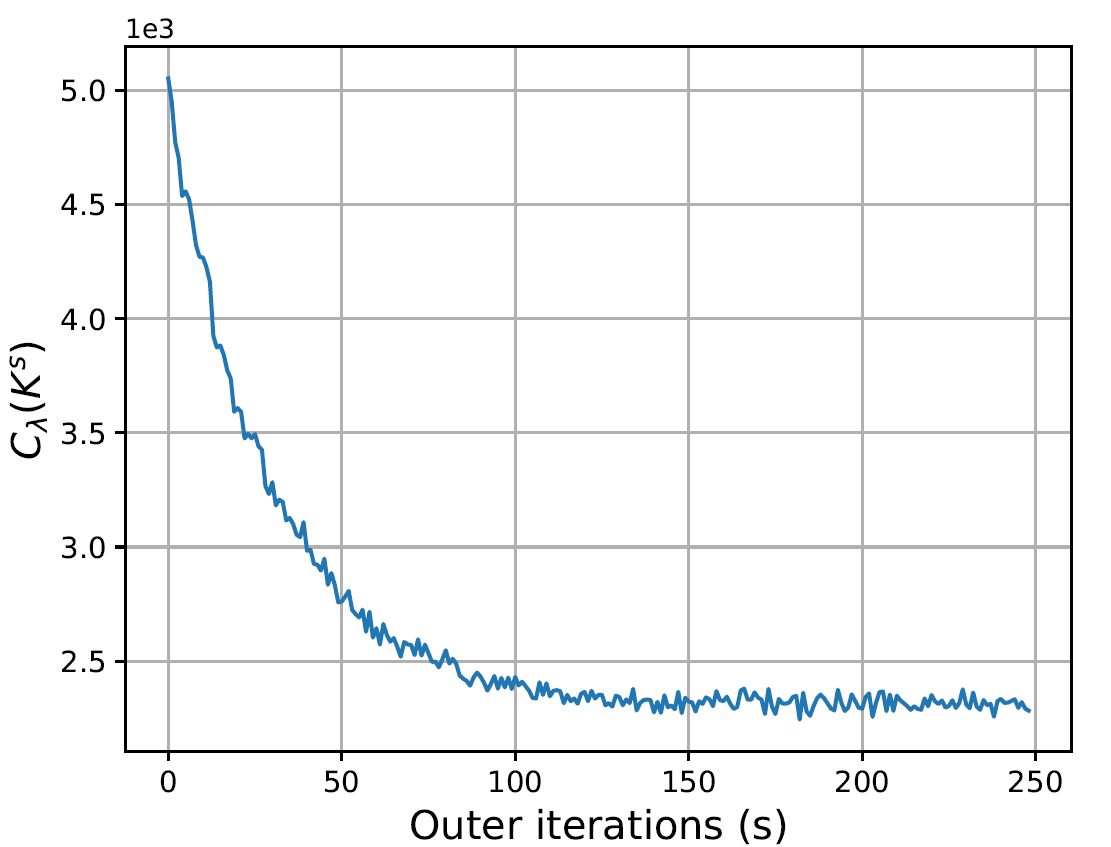}
\caption{}
\label{fig:convergence}
\end{subfigure}%
~ 
\begin{subfigure}[t]{0.32\textwidth}
\centering
\includegraphics[scale=0.41]{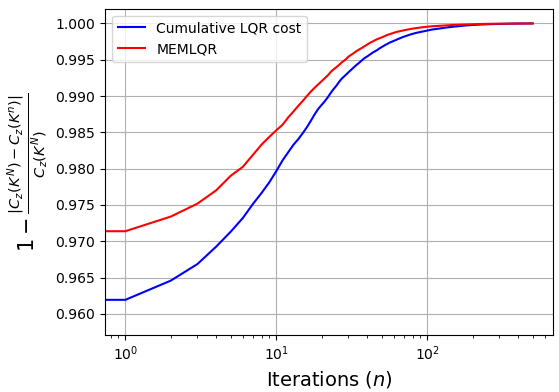}
\caption{}
\label{fig:accuracy}
\end{subfigure}
~
\begin{subfigure}[t]{0.32\textwidth}
\centering
\includegraphics[scale=0.18]{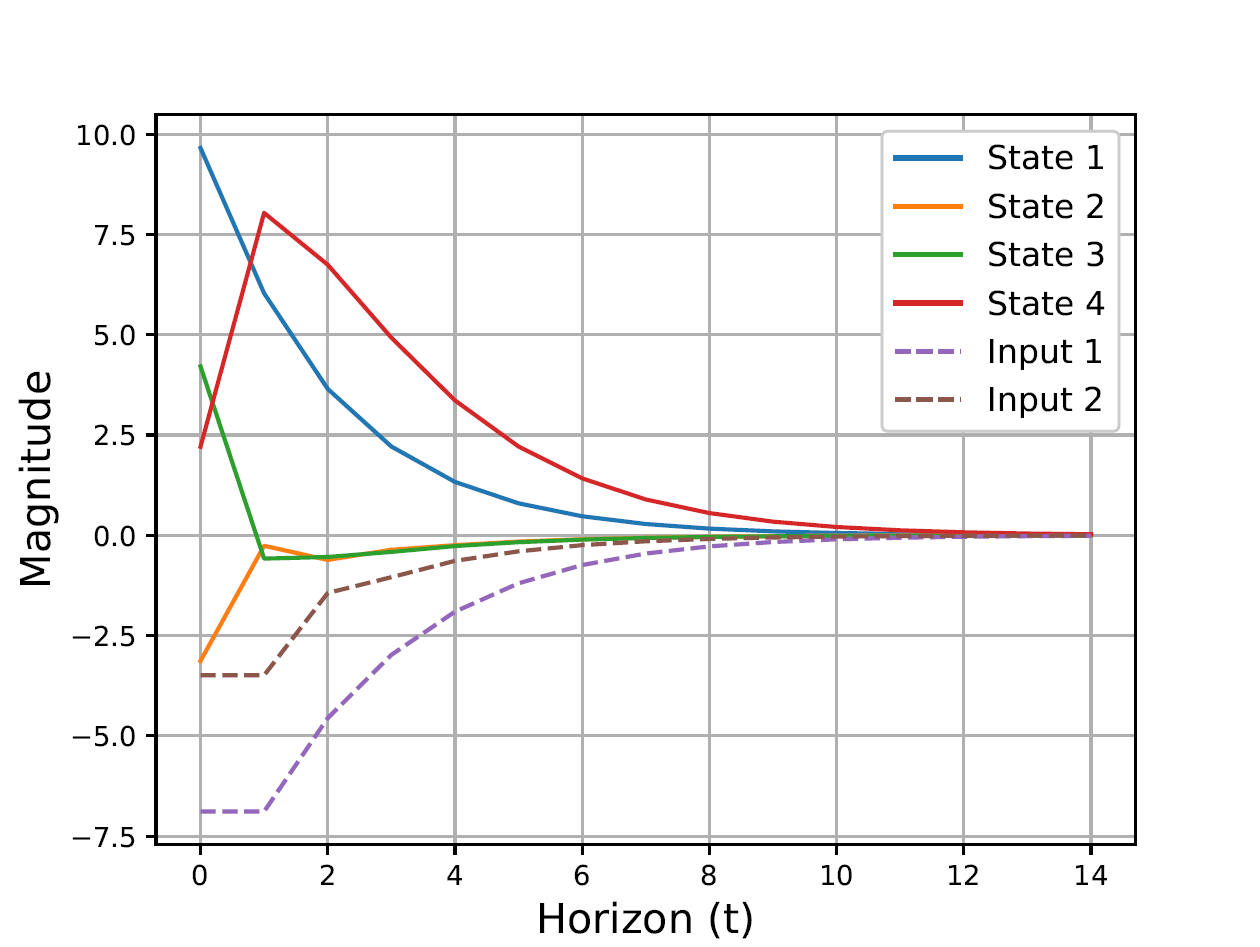}
\caption{}
\label{fig:trajectories}
\end{subfigure}
\caption{Depiction of properties of Algorithm~\ref{algo:dis_com}: (a) Convergence of the MEMLQR algorithm: evolution of the Moreau envelope regularized cost $C_\lambda(K^s)$ as the number of outer iterations $(s)$ increase. (b) The evolution of accuracy $\pbrack{1-{|C_{z}(K^N)-C_{z}(K^n)|}/{C_{z}(K^N)}}$ of the policy generated for an unseen realization $z$ by a model-based policy gradient framework when initialized using a policy obtained by minimizing the total LQR cost $C(\cdot)$ and by using the $K^S$ generated by the MEMLQR algorithm. Note that the cost calculation here is over 50 randomly initial states. (c) State and input trajectories of the unseen realization $z$ for the estimate of the optimal policy $(K^N)$ generated after $N=250$ iterations of the model-based policy gradient algorithm, which was initialized at $K^S$ generated by the MEMLQR algorithm.}
\label{fig:properties}
\end{figure*}

\subsection{Comparison to the MAML approach}
\label{ssec:comparison}

To empirically compare the performance of the MAML-based approach to that of ours, we borrowed the example presented in~\cite{LFT-DZ-JA-HW:24}. We generated ten different system realizations and obtained meta policies using the MAML-LQR algorithm from~\cite{LFT-DZ-JA-HW:24} and our MEMLQR algorithm with three different values, i.e., $\reg\in\setdef{0.02, 0.2, 2}$. These policies are evaluated by comparing the convergence characteristics while performing model-free policy gradient for three previously unseen system realizations. It can be observed from Fig.~\ref{fig:alltask} that at initialization, the MEMLQR policy has a lower cost than the MAML-LQR policy for all three realizations, thus leading to a faster improvement in the cost incurred.

\begin{figure*}[t!]
\centering
\begin{subfigure}[t]{0.32\textwidth}
\centering
\includegraphics[scale=0.25]{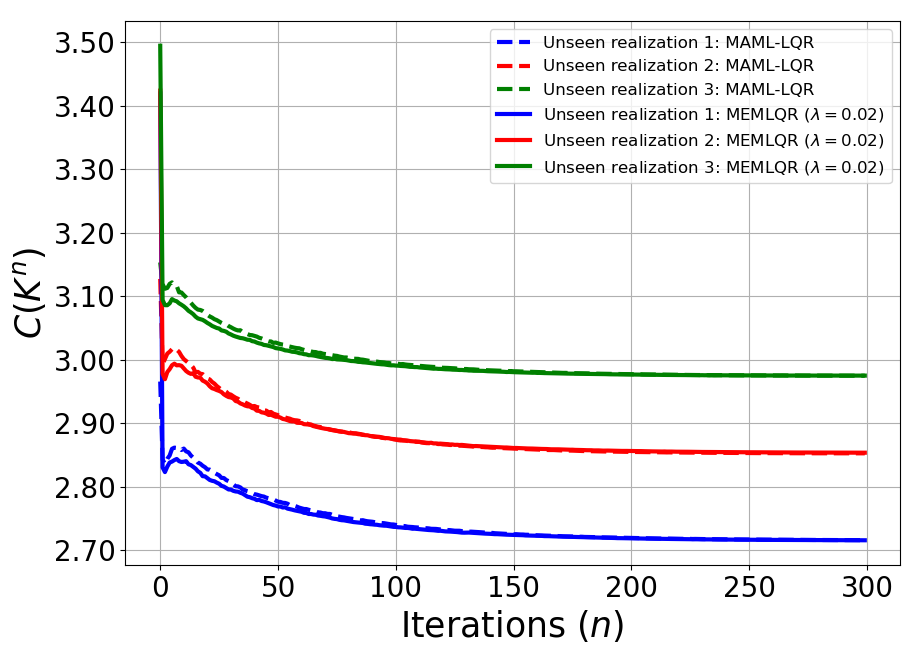}
\caption{}
\label{fig:melqr002}
\end{subfigure}%
~ 
\begin{subfigure}[t]{0.32\textwidth}
\centering
\includegraphics[scale=0.25]{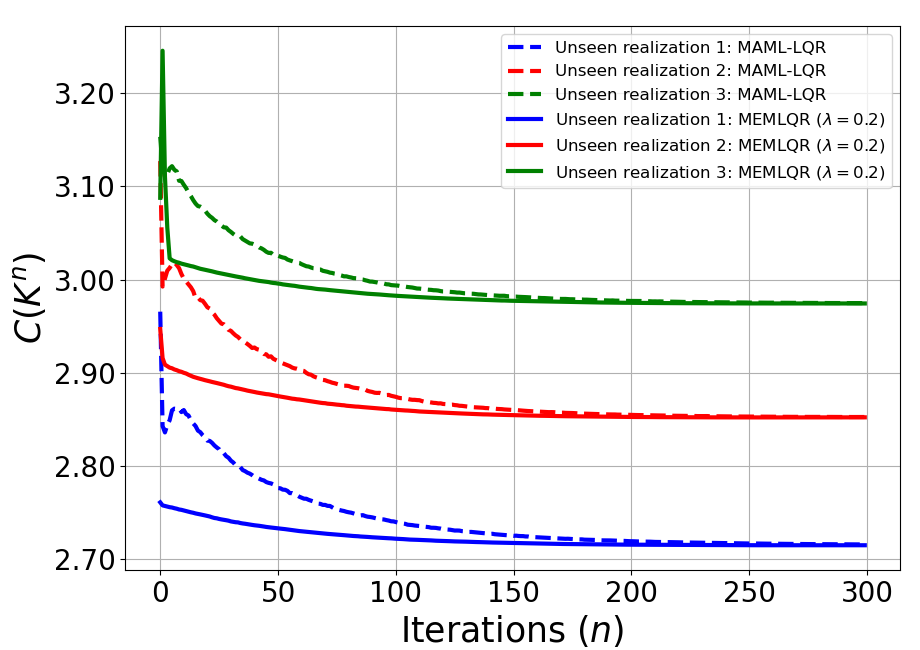}
\caption{}
\label{fig:melqr02}
\end{subfigure}
~
\begin{subfigure}[t]{0.32\textwidth}
\centering
\includegraphics[scale=0.25]{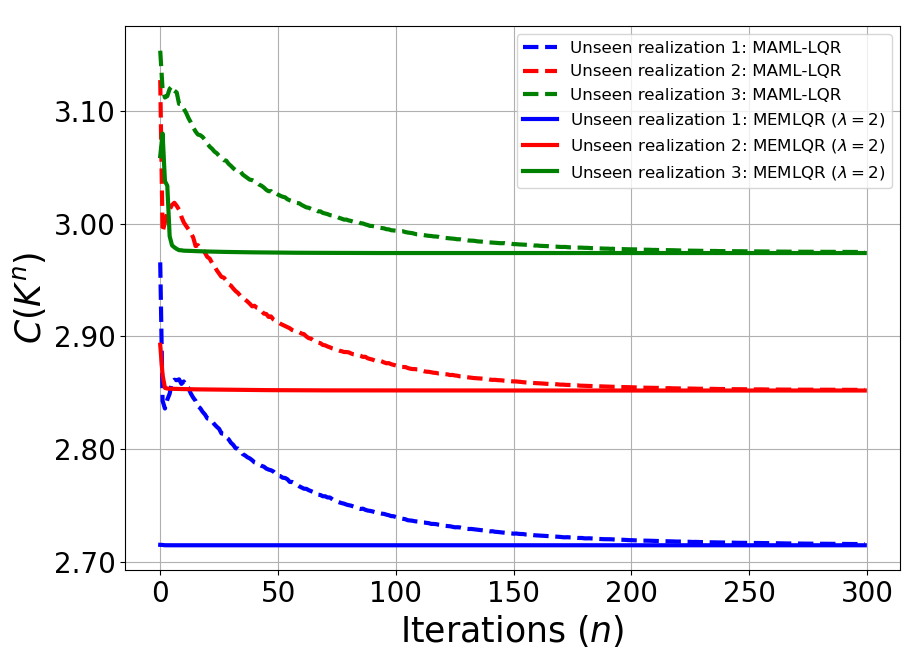}
\caption{}
\label{fig:melqr2}
\end{subfigure}
\caption{ Convergence to the optimal policy after initialization using the policy generated by MAML-LQR and MEMLQR ($\reg = 0.02,0.2,2$) for three random system realizations in a model-free setting. It can be observed that the cost incurred by policy generated by the MEMLQR approach is closer to the optimal value initially, thus aiding in faster convergence. Also, it may be noted that for higher values of $\reg$, the cost incurred is lower.}
\label{fig:alltask}
\end{figure*}

%

\section{Conclusion}
\label{sec:conclusions}
In this article, we explored personalization using Moreau envelopes to obtain an initialization for policy gradient algorithms (model-based/free) when dealing with uncertain linear systems, given that there is access to a finite number of uncertainty realizations. An iterative first-order framework, Algorithm~\ref{algo:dis_com}, was presented to optimize this personalized cost function in a client-server setup, and its analysis showed the algorithm's convergence to a first-order stationary point. Numerical experiments were provided to empirically show the advantage of designing and using the proposed framework's meta-policy rather than the total cost. We compared the results of our approach to that of an approach based on the MAML framework and empirically showed that our approach incurred a lower cost for unseen scenarios. 

\bibliographystyle{ieeetr}
\bibliography{refs}
\end{document}